\documentclass[10pt,twocolumn,twoside]{IEEEtran}

\usepackage{amssymb,amsmath}
\usepackage{amsmath,amsfonts,amssymb,amscd,latexsym,enumerate,bbm,stfloats}
\usepackage{amsfonts,mathrsfs}
\usepackage{multicol}
\usepackage{amsthm}
\usepackage{bbold}
\usepackage{dsfont}
\usepackage{enumerate}
\newcommand{\bbm}{\begin{bmatrix}}
\newcommand{\ebm}{\end{bmatrix}}
\usepackage{color,soul}

\newcommand{\bse}{\begin{subequations}}
\newcommand{\ese}{\end{subequations}}
\newcommand{\ones}{\mathds{1}}

\DeclareMathOperator{\col}{col}
\DeclareMathOperator{\im}{im}
\DeclareMathOperator{\diag}{diag}
\DeclareMathOperator{\bdiag}{blockdiag}

\newcommand{\dst}{\displaystyle}

\newcommand{\R}{\mathbb{R}}

\newcommand{\calH}{\ensuremath{\mathcal{H}}}

\newcommand{\calG}{\ensuremath{\mathcal{G}}}
\newcommand{\calV}{\ensuremath{\mathcal{V}}}
\newcommand{\calE}{\ensuremath{\mathcal{E}}}
\newcommand{\calN}{\ensuremath{\mathcal{N}}}

\newcommand{\ew}{\hspace*{\fill} $\square$\noindent}

\def\be{\begin{equation}}
\def\ee{\end{equation}}
\def\ba{\begin{array}}
\def\ea{\end{array}}
\def\eqa{\begin{eqnarray}}
\def\eqe{\end{eqnarray}}

\newtheorem{example}{Example}
\newtheorem{assum}{Assumption}
\newtheorem{definition}{Definition}
\newtheorem{theorem}{Theorem}

\newtheorem{lem}{Lemma}
\newtheorem{corollary}{Corollary}
\newtheorem{proposition}{Proposition}
\newtheorem{remark}{Remark}

\def\QEDopen{{\setlength{\fboxsep}{0pt}\setlength{\fboxrule}{0.2pt}\fbox{\rule[0pt]{0pt}{1.3ex}\rule[0pt]{1.3ex}{0pt}}}}
\def\QED{\QEDopen} 
\newenvironment{pfof}[1]{\vspace{1ex}\noindent{\itshape Proof of
    #1:}\hspace{0.5em}} {\hfill\QED\vspace{1ex}}

\usepackage[prependcaption,colorinlistoftodos]{todonotes}

\newcommand{\cdp}[1]{{\color{black}#1}}
\newcommand{\cdpbis}[1]{{\color{black}#1}}
\newcommand{\cdpter}[1]{{\color{black}#1}}

\newcommand{\n}[1]{\color{black}{#1}}

\newcommand{\nm}[1]{{\color{black} #1}}
\newcommand{\nmo}[1]{{\color{black} #1}}
\usepackage{booktabs}

\usetikzlibrary{decorations.markings}         
\tikzstyle{vertex}=[circle, shading = ball, ball color = white!100!white, minimum size = 15pt, draw, inner sep=0pt]  
\newcommand{\vertex}{\node[vertex]}           
\newcommand{\weight}[1]{{\footnotesize $\mathit{#1}$}}

\title{A feedback control algorithm to steer networks to a Cournot-Nash equilibrium
}
\author{Claudio De Persis \;\and \; Nima Monshizadeh
\thanks{Claudio De Persis and Nima Monshizadeh are with the Engineering and Technology Institute, University of Groningen, The Netherlands, {\tt\small c.de.persis@rug.nl, n.monshizadeh@rug.nl}}}

\begin{document}

\maketitle

\begin{abstract}
We propose a distributed feedback control that steers a dynamical network to a prescribed equilibrium corresponding to the so-called Cournot-Nash equilibrium.
The network dynamics considered here are a class of passive nonlinear second-order systems, where production and demands act as external inputs to the systems. While productions are assumed to be controllable at each node, the demand is determined as a function of local prices according to the utility of the consumers.
Using reduced information on the demand, the proposed controller guarantees the convergence of the closed loop system to the optimal equilibrium point dictated by the Cournot-Nash competition. 
\end{abstract}



\section{introduction}
In recent years there has been a renewed interest in controllers that can steer a given dynamical system to an optimal steady state, mostly motivated by research connected to power networks, where, for instance, given a certain demand, the problem of dynamically adjusting the generation in order to satisfy the demand while fulfilling an optimal criterion at steady state, e.g., minimizing the generation costs or maximizing the social welfare, has been formulated and addressed with different approaches. Other application domains of interests are flow and heat networks, data centers as well as logistic systems. 

Loosely speaking, the proposed approaches to dynamically control networks while fulfilling steady-state optimality criteria  can be classified in two categories. One relies on primal-dual gradient algorithms, which solve  the optimization problem, and apply on-line the computed control input to the physical network, possibly taking into account feedback signals coming from the network for improved robustness \cite{zhang-pap-aut15, Li-Zhao-Chen-TCNS16, stegink-et-alt-ac17, shiltz-et-al}. This approach returns control algorithms that can handle general convex objective functions and constraints but typically requires either a good knowledge of the demand, or a number of measurements of network variables, sometimes more than what is really available in practice.   

A second category relies on internal-model based controllers \cite{dorfler-et-al-tcns16, trip-et-al-aut16, andreasson2014distributed}, popularized under the acronym DAPI (distributed averaging proportional integral controllers) \cite{simpson-porco-aut13}, which can typically deal with linear-quadratic cost functions only, but can on the other hand  tackle uncertainty in the power demand. 

This paper aims at contributing to the second category of results allowing for some degree of uncertainty in the demand and  considering an economic objective 
which is different from the ones considered so far -- economic optimal dispatch, implemented via a distributed \cite{dorfler-et-al-tcns16, trip-et-al-aut16} or a semi-decentralized control architecture \cite{dorfler-grammatico-aut17}. In fact, our interest is to design controllers at the producers that aim at maximizing their profit according to a Cournot model of competition \cite{mas1995microeconomic}. 

Other works are available that have solved problems different from the economic optimal  dispatch problem, e.g., \cite{ashish,stegink-et-al}, where algorithms solving the optimal power generation model under a Bertrand model of competition and a primal-dual setting have been proposed.  

Cournot models of competition, and the resulting Cournot-Nash equilibrium \cite{johari2005efficiency}, \cite{tsitsiklis-yu}, \cite{bose-cdc2014} are very well studied topics in game theory. Pseudo-gradient dynamical algorithms that  converge to Cournot-Nash equilibria have  been extensively investigated in the literature \cite{nagurney.book93, 
grammatico-tcns17, yi-pavel}. 

What differentiates our results from the existing one are two features. First, as highlighted before,  we devise a feedback control algorithm that does not rely on the exact knowledge of the demand in the network. Second, this algorithm is interconnected in a feedback loop with a given physical network, and the resulting closed-loop system is analyzed as a whole, showing that the physical network converge to an equilibrium at which the controller output equals the value of the Cournot-Nash equilibrium.  Since the controller state  variable has the interpretation of a  price, we at times refer to the feedback control also as a pricing mechanism. 

Although the motivation for this investigation was inspired by problems in power networks, in order to convey the results to an audience that is not necessarily interested to power networks, we decided to present our results for a class of nonlinear second-order passive dynamical networks \cite{arcak}, in which power networks fall after suitable modifications.

The rest of the paper is organized as follows. In the next section, we recall basic concepts and results about the Cournot model of competition. Section \ref{sec:3} contains the main results of the paper, namely the design and analysis of a distributed feedback controller steering the closed-loop system towards a Cournot-Nash equilibrium. A numerical case study is discussed in Section \ref{s:simulation}. Conclusions are drawn in Section \ref{sec:4}. 
\section{Cournot Competition}
%
In this section, we revisit a few results  about the Cournot model of competition (\cite{johari2005efficiency,mas1995microeconomic}), which are used to determine the optimal triple of  supply, demand and price.  
This characterization is instrumental to formulate our main results in the next section, where a dynamic controller is designed to steer a given network to such optimal Cournot triple. 

In a model of Cournot competition, $n$ producers produce a homogeneous good that is demanded by $m$ consumers. The price $p$ of the good is assumed to be determined by the good productions, while the consumers are price takers, a scenario motivated by having few producers and many consumers.


Each producer $i$ aims at maximizing its  profit $\Pi_{gi}\cdp{:\mathbb{R}_{\ge 0}^n\to \mathbb{R}}$, given the production of other firms. The profit is defined by the objective function
\[
\Pi_{gi}(P_{gi}, P_{-gi}) = p(P_g) P_{gi} - C_{gi}\cdp{(P_{gi})},\, i\in I:=\{1,2,\ldots, n\},
\]
where $P_{gi}\in \mathbb{R}_{\ge 0}$ is the good production by producer $i$,  $P_g={\rm col}(P_{g1},\ldots, P_{gn})\in \mathbb{R}^n$ is the vector of all the productions, $\cdp{P_{-gi}}\in \mathbb{R}^{n-1}$ is the vector obtained by removing the $i$th element of $P_g$, {$p:\R^n \rightarrow \R$ is the price function,} and $C_{gi}:\mathbb{R}_{\ge 0}\to\mathbb{R}$ is a function satisfying the following assumption \cite{johari2005efficiency}:
\begin{assum}\label{assmpt:cost.function}
For each $i$, the cost function $C_{gi}$ is convex, non-decreasing and continuously differentiable for $P_{gi}\in \mathbb{R}_{\ge 0}$. \cdp{Moreover, $C_{gi}(0)=0$. }
\end{assum}

Given a price {$\overline p\in \mathbb{R}$}, each consumer $j$ wants to maximize its utility, described by the function
\[
\Pi_{dj}(P_{dj},{\overline p}) = U_j (P_{dj}) - {\overline p}P_{dj},\; j\in J:=\{1,2,\ldots, m\}.
\]
where $P_{dj}\in \n{\mathbb{R}_{\geq 0}}$ is the good demand by consumer $j$ and $U_{j}:\mathbb{R}_{\ge 0}\to \mathbb{R}$ is a 
continuously differentiable function. 
{To ease the notation, in the sequel, we denote $\overline p$ simply by $p$.}
{Now, given a price {$p\in \mathbb{R}$}}, we {consider} the utility maximisation problem
\be\label{ut.max.pbm}
\sup_{P_{dj}\ge 0} \Pi_{dj}(P_{dj}, {p})
\ee

\begin{assum}\label{a:Uj}
For each $j\in J$ the utility function $U_j:\mathbb{R}_{\ge 0}\to \mathbb{R}$ is continuously differentiable and \n{strictly} concave. Moreover,  $U_j':\mathbb{R}_{\ge 0}\to \mathbb{R}$ { satisfies $\lim_{P_{dj}\rightarrow +\infty} U'_j(P_{dj})=-\infty$ and $U_{j}'(0)>0$.}
\end{assum}
%
%
\n{\subsection{Utility Maximisation}}
The utility maximisation problem admits the following solution:
\begin{proposition}\label{first.order.optimality.condition}
Let  Assumption \ref{a:Uj} hold. 
Then, $P_{dj}$ is solution to  \eqref{ut.max.pbm} if {and only if}
\be\label{e:optimal.Pd}
P_{dj}=  \pi_{j}(p),
\ee
where $\pi_j: \R \rightarrow \R_{\geq 0}$ is defined as
\be\label{pi.j}
\pi_j(\lambda)=\left\{
\ba{ccr}
(U_{j}')^{-1}(\lambda) & \textrm{if} & \lambda \nm{<} U_j'(0)\\[2mm]
0 & \textrm{if} & \lambda \nm{\geq} U_j'(0).
\ea
\right.
\ee
\end{proposition}

\begin{proof}
The proof descends  from the KKT conditions. 
\end{proof}

\subsection{Supply-demand matching}


{As will be discussed in Section \ref{sec:3}, we are interested in a supply-demand balancing condition, where the total generation is equal to the total demand, namely}
 \be\label{balance0}
\mathds{1}^\top P_g=\mathds{1}^\top P_d,
\ee
where $P_d=(P_{d1},\ldots, P_{dm})$ and $\mathbb{1}$ denotes the vector of all ones of suitable size.
{Under this balancing condition, the price can be written as a function of the total production, as stated next.}
%
%

\begin{proposition}\label{prop.price}
Let Assumption \ref{a:Uj} and the balance equation \eqref{balance0} hold, \cdp{with $\ones^\top P_g\ge 0$}. 
Then, there exists a continuous function $u:\mathbb{R}_{\ge 0}\to \mathbb{R}$ satisfying
\be\label{e:price-general}
{u(\mathds{1}^\top P_g)=p(P_g)}
\ee
with the following properties: (i) $u(0)>0$; (ii) $u$ is strictly decreasing; (iii) $\lim_{q\to +\infty} u(q)=-\infty$.
\end{proposition}

\begin{proof}
In view of Assumption \ref{a:Uj}, the functions $\pi_{j}(\lambda)$ in \eqref{e:optimal.Pd} are continuous  with range equal to 
the whole $\mathbb{R}_{\ge 0}$. 
Each function 
$\pi_{j}(\lambda)$
is strictly decreasing on the interval  
$(-\infty, U_j'(0)]$ 
and identically zero on the interval \cdp{$[U_j'(0), +\infty)$.} Define the function {$\pi(\lambda) {:=}\sum_{j\in J} \pi_{j}(\lambda)$.} This is a continuous  function  with  range equal to 
the whole $\mathbb{R}_{\ge 0}$.
Moreover it is strictly decreasing on 
$(-\infty, \max_j U_j'(0)]$ 
and identically zero on the interval 
$[\max_j U_j'(0), +\infty)$. 
Let $q:=\pi(\lambda)$.
Then, the function $\pi$ can be inverted, on the interval {where $q\in \R_{>0}$}, to obtain $\lambda=\pi^{-1}(q)$.
%
Now, we define $u:\mathbb{R}_{\ge 0}\to \mathbb{R}$ as 
\be\label{e:uq}
u(q)=\left\{
\ba{rcl}
\pi^{-1}(q) & if & q>0\\[1mm]
\pi^{-1}(0^+) & if & q=0,
\ea
\right.
\ee
where $\lim_{q\to 0^+} \pi^{-1}(q) =:\pi^{-1}(0^+)$.
By construction, $u$ is a continuous function, which is strictly decreasing and satisfies the properties (i),  (ii), (iii) of the statement. 
{Moreover, by substituting $q=\ones^\top P_g$ in \eqref{e:uq}, using \eqref{e:optimal.Pd} together with the balancing condition \eqref{balance0}, and noting monotonicity of $u$, we obtain the equality \eqref{e:price-general}.}
\end{proof}


\begin{example}
Consider the case of $2$ consumers with linear-quadratic utility functions 
\[
U_j(P_{dj})=-\frac{1}{2} P_{dj} Q_{dj} P_{dj}+ b_{dj} P_{dj},
\]
where $Q_{dj}, b_{dj}>0$, $j=1,2$. 
{For each $j$, we have}
\be\label{e:pi-lq}
\cdpbis{\pi_{j}(\lambda)}
=\left\{
\ba{rcl}
Q_{dj}^{-1}(b_{dj}-\lambda) & \textrm{if} & \lambda < b_{dj}\\
0 & \textrm{if} & \lambda \ge b_{dj}.
\ea
\right.
\ee
Let $b_{d1}<b_{d2}$. Then 
\[
\pi(\lambda)= \left\{\ba{lll}
\dst\sum_{j=1,2} Q_{dj}^{-1}(b_{dj}-\lambda) & if & \lambda \le b_{d1}\\[2mm]
Q_{d2}^{-1}(b_{d2}-\lambda) & if &  b_{d1}{<} \lambda \le b_{d2} \\
0 & if & \lambda {>} b_{d2}
\ea
\right. 
\]
and 
for $q=\pi(\lambda)$, $q>0$, we have
\[
\pi^{-1}(q)= \left\{\ba{lll}
-Q_{d2} q + b_{d2}
& if & 0<q\le Q_{d2}^{-1}(b_{d2}-b_{d1})\\
-\alpha q +\beta 
& if & q\ge Q_{d2}^{-1}(b_{d2}-b_{d1})
\ea
\right. 
\]
where
\[
 \beta = \frac{\sum_{j=1,2} \frac{b_{dj}}{Q_{dj}}}{\sum_{j=1,2} \frac{1}{Q_{di}}}, \quad \alpha= \frac{1}{\sum_{j=1,2} \frac{1}{Q_{di}}}.
\]
It then follows that 
\[
u(q)= \left\{\ba{lll}
-Q_{d2} q + b_{d2}
& if & 0\le q \le Q_{d2}^{-1}(b_{d2}-b_{d1})\\
-\alpha q +\beta 
& if & q\ge Q_{d2}^{-1}(b_{d2}-b_{d1}).
\ea
\right. 
\]
\end{example}
\ew

In view of \eqref{e:optimal.Pd} and \eqref{e:price-general}, we conclude that the optimal demands are given by
\be\label{e:optimal.Pd.3}
P_{dj}=\left\{
\ba{lcl}
(U_{j}')^{-1}(u(\mathds{1}^\top P_g)) & \textrm{if} & u(\mathds{1}^\top P_g) < U_j'(0)\\
0 & \textrm{if} & u(\mathds{1}^\top P_g) \ge U_j'(0).
\ea
\right.
\ee
By construction, the optimal demand $P_d$ detailed above satisfies $\ones^\top P_d=\ones^\top P_g$. 

\n{\subsection{Profit Maximisation: the Cournot-Nash equilibrium} }
Motivated by the discussion before, we consider a Cournot game consisting of the set of producers $I$ each one aiming at solving the maximisation problem 
\be\label{ump}
\max_{P_{gi}\ge 0}  \Pi_{gi}(P_{gi}, P_{-gi}),
\ee
with
\be\label{eq:profit}
\Pi_{gi}(P_{gi}, P_{-gi}) = 
\cdpter{p(P_g)}
P_{gi} - C_{gi}\cdp{(P_{gi})}.
\ee
More formally, we define the Cournot game
as follows:
\medskip{}
\begin{definition}
A Cournot game  $CG(I, (\Pi_{gi}, i\in I))$ consists of 
\begin{enumerate}[i)]
\item A set $I$ of producers (or players);
\item  A strategy $P_{gi} \in \mathbb{R}$ for each producer $i\in I$;
\item The convex and closed set  $\mathbb{R}^n_{\ge 0}$ of allowed strategies
$P_g=(P_{g1}, \ldots, P_{gn})$;
\item A payoff function $\Pi_{gi}(P_{gi}, P_{-gi})$, where for $P_g\in \mathbb{R}^n_{\ge 0}$, $\Pi_{gi}(P_{g})$ is continuous in $P_{g}$ and concave in $P_{gi}$ for each fixed $P_{-gi}$.  
\end{enumerate}
\end{definition}

The Cournot-Nash equilibrium is defined next \cite{johari2005efficiency}:
\begin{definition}
A Cournot-Nash equilibrium of the game $CG(I, (\Pi_{gi}, i\in I))$ is a vector $P_g^\star\in \mathbb{R}^n_{\ge 0}$ that for each $i\in I$ satisfies
\[
\Pi_{gi}(P_{gi}^\star, P_{-gi}^\star)\ge   \Pi_{gi}(P_{gi}, P_{-gi}^\star)
\]
for all $P_{gi}\in \mathbb{R}_{\ge 0}$.
\end{definition}

The existence of a Cournot-Nash equilibrium is a consequence of a well-know result on concave games due to  \cite{rosen1965existence}. Let us recall the definition of a concave game.
\begin{definition}
A concave game consists of
\begin{enumerate}[i)]
\item A set $I$ of players;
\item A strategy $x_i \in \mathbb{R}^{p_i}$ for each player $i\in I$;
\item The convex, closed and bounded set $R\subset \mathbb{R}^p$ of allowed strategies $x=(x_1, \ldots, x_n)$, with $p=\sum_{i=1}^n p_i$;
\item A payoff function $\varphi_i(x_i, x_{-i})$ for each player $i\in I$, where for $x\in S$, $\varphi_i (x)$ is continuous in $x$ and concave in $x_i$ for each fixed $x_{-i}$, with $S=P_1\times \ldots \times P_n$ and $P_i$ is the projection of $R$ on $\mathbb{R}^{m_i}$.  
\end{enumerate}
\end{definition}

A difference between a Cournot and a concave game is the lack of a bounded  set of bounded strategies for the Cournot game ($\mathbb{R}^n_{\ge 0}$ is clearly unbounded). However, following \cite[Proposition 2]{johari2005efficiency}, it can be shown that solving \eqref{ump} is equivalent to solving 
\be\label{ump.eq}
\max_{0\le P_{gi}\le \bar P_{gi}}  \Pi_{gi}(P_{gi}, P_{-gi}),
\ee
with $\bar P_{gi}$ some positive constant. As a matter of fact, $\Pi_{gi}(P_{gi}, P_{-gi}) = u(\mathds{1}^\top P_g) P_{gi} - C_{gi}\cdp{(P_{gi})}$ is zero at $P_{gi}=0$,  and by Assumptions \ref{assmpt:cost.function}, \ref{a:Uj} and Proposition \ref{prop.price}, there exists $\bar P_{gi}>0$ such that $\Pi_{gi}(\bar P_{gi}, P_{-gi})=0$ and $\Pi_{gi}(P_{gi}, P_{-gi})<0$ for $P_{gi}\ge \bar P_{gi}$. Hence, in the case 
of \eqref{ump.eq}, $R=S=[0, \bar P_{g1}]\times \ldots \times [0, \bar P_{gn}]$. Moreover, the payoff function $\Pi_{gi}(P_{gi}, P_{-gi})$ satisfies the properties of  a concave game by Assumptions \ref{assmpt:cost.function}, \ref{a:Uj} and Proposition \ref{prop.price}. Therefore the game defined by \eqref{ump.eq}, or equivalently by \eqref{ump}, is a concave game. It can then be concluded by \cite[Theorem 1]{rosen1965existence} that a Cournot-Nash equilibrium exists. 

\begin{theorem}\label{thm:Nash-rosen1965existence}
{\rm \cite[Theorem 1]{rosen1965existence}} Under Assumptions \ref{assmpt:cost.function} and \ref{a:Uj} there exists a Cournot-Nash equilibrium $P_g^\star$.
\end{theorem}

\subsection{Linear-quadratic utility and cost functions}
In this subsection and for the remainder of the paper, we restrict the cost and utility functions of producers and consumers to linear-quadratic functions, namely 
\be\label{e:Cg-LQ}
C_{gi}(P_{gi}) = \frac{1}{2} P_{gi} Q_{gi} P_{gi} + b_{gi} P_{gi}, 
\quad \n{b_{gi},\;} Q_{gi}>0,
\ee
\be\label{e:U-LQ}
U_j(P_{dj})=-\frac{1}{2} P_{dj} Q_{dj} P_{dj}+ b_{dj} P_{dj}, \quad\n{b_{dj, \;}} Q_{dj}>0.
\ee

Proposition \ref{first.order.optimality.condition} is specialized as follows:
\begin{corollary}\label{first.order.optimality.condition.lq}
{The scalar $P_{dj}$  {is a solution to \eqref{ut.max.pbm} with $U_j$ as in \eqref{e:U-LQ}}
\cdp{if and only if}
\be\label{e:optimal.Pd.lq}
P_{dj}=\left\{
\ba{llll}
Q_{dj}^{-1}(b_{dj}-p) & \textrm{if} & p \hspace{-0.25cm}&< b_{dj}\\
0 & \textrm{if} & p  \hspace{-0.25cm}&\ge b_{dj}
\ea
\right.
\ee
}
\end{corollary}


We are particularly interested in the case where all producers and consumers enter the market, that is $P_{gi}, P_{dj}>0$ for all $i\in I$ and $j\in J$. 
{Conditions under which this case occurs are formalised next.
}
{
\begin{lem}\label{l:equivalence-demand}
Let the  utility functions of  consumers be given by \eqref{e:U-LQ} and consider 
the utility maximization problem \eqref{ut.max.pbm}.
Then, the following statements are equivalent:
\begin{enumerate}
\item There exists $P_d\in \R^n_{>0}$ solution to \eqref{ut.max.pbm}. 
\medskip{}

\item $p<\underline{b}_d$, where $\underline{b}_d:=\min\{b_{dj}:j\in J\}$.

\medskip{}

\item The vector $P_d$ given by 
\be\label{positive.demand}
P_d=Q_d^{-1}(b-\ones p),
\ee
is {the unique} solution to \eqref{ut.max.pbm}, and $p\neq \underline{b}_d$. 
\end{enumerate}

\end{lem}
}
%
%

\begin{proof}
See Appendix.
\end{proof}

\medskip
\cdpter{We note that, given the strictly positive demand \eqref{positive.demand} in Lemma \ref{l:equivalence-demand}, the expression can be inverted to obtain}
{

\be\label{e:price-affine*}
p= \beta^* -\alpha^* \ones^\top P_d
\ee
where
\be\label{e:alpha-beta}
 {\beta^*} := \dst\frac{\dst\sum_{j\in J} \frac{b_{di}}{Q_{di}}}{\dst\sum_{j\in J} \frac{1}{Q_{di}}}=\frac{\ones^\top Q_d^{-1} b_d}{\ones^\top Q_d^{-1}\ones }, \quad {\alpha^*}:= \dst\frac{1}{\dst\sum_{j\in J} \frac{1}{Q_{di}}}=\dst\frac{1}{\ones^\top Q_d^{-1}\ones}, 
\ee
and
$
b_d := {\rm col}(b_{d1}, \ldots, b_{dn}), \quad Q_d ={\rm diag} (Q_{d1}, \ldots, Q_{dn}),
$
$
b_g := {\rm col}(b_{g1}, \ldots, b_{gn}), \quad Q_g ={\rm diag} (Q_{g1}, \ldots, Q_{gn}).$



\vspace*{-0.3cm}
\[
{\color{white}***}
\]

We turn now our attention to the producers. Let the price function \cdpter{in \eqref{ump}} admit the affine form  
\be\label{e:price-affine}
p(P_g)=\beta- \alpha \ones^\top P_g
\ee
for some scalars $\alpha, \beta>0$, which is motivated by Proposition \ref{prop.price} specialized to the case of linear-quadratic cost functions with strictly positive generation and demand.
In the next section, the scalars $\alpha$ and $\beta$ in \eqref{e:price-affine}  will be set to those in \eqref{e:alpha-beta}, as we are interested in the supply-demand matching condition \eqref{balance0}. This will be made more explicit in Theorem \ref{p:optimal}.

Since Assumptions \ref{assmpt:cost.function}  and \ref{a:Uj} are satisfied in the case of linear-quadratic functions, Theorem \ref{thm:Nash-rosen1965existence} holds and  a Cournot-Nash equilibrium exists. Then,
the computation of the Cournot-Nash equilibrium $P_g^\star$ descends from the optimization problem 
\be\label{game.lq}
P_{gi}^\star \in 
\dst\arg\max_{P_{gi}\ge 0} \Pi_{gi}(P_{gi}, P_{-gi}^\star)
\ee
where, in view of \eqref{eq:profit}, \eqref{e:Cg-LQ}, \eqref{e:price-affine}, 
\[\ba{rl}
\Pi_{gi}(P_{gi}, P_{-gi}) =& (\beta -\alpha  P_{gi} -\alpha \mathds{1}^\top P_{-gi})P_{gi}\\
&-\frac{1}{2} P_{gi} Q_{gi} P_{gi} - b_{gi} P_{gi}.
\ea\]
{The conditions} under which the parabola 
$(\beta -\alpha  P_{gi} -\alpha \mathds{1}^\top P_{-gi})P_{gi}
-\frac{1}{2} P_{gi} Q_{gi} P_{gi} - b_{gi} P_{gi}$ has a {nonnegative} maximizer 
can be formalized as follows:
\begin{proposition}\label{first.order.optimality.condition.game.lq}
$P_{gi}$ is a solution to  \eqref{game.lq} if and only if
\be\label{e:optimal.Pg}
P_{gi}=\left\{
\ba{rcl}
\frac{\gamma_i(P_{-gi})}{2\alpha +Q_{gi}} & \textrm{if} & \gamma_i(P_{-gi})>0\\
0 & \textrm{if} & \gamma_i(P_{-gi}) \le 0
\ea
\right.
\ee
with 
\[
\gamma_i(P_{-gi})
: = \beta-b_{gi} - \alpha \sum_{j\ne i:P_{gj}>0} P_{gj}.
\]
\end{proposition}
The proof 
is omitted.

%

Recall that we are interested in the case where every
producer contributes a strictly positive production, i.e. $P_{g}\in \R^n_{>0}$. This brings us to the following lemma:
\begin{lem}\label{cor:opt.Pg}
The following statements are equivalent:
\begin{enumerate}
{\item 
\cdpter{There exists $P_g\in \R^n_{>0}$ solution to  \eqref{game.lq}.}
}
\medskip{}
\item The vector $P_g$ given by 
\be\label{Pg.opt}
P_g= (\alpha (I+\mathds{1}  \mathds{1}^\top)  +Q_g )^{-1}(\beta \ones-b_g), 
\ee  
{is \cdpter{the unique} solution to  \eqref{game.lq}} and \cdpter{$\cdpter\beta-\alpha \ones^\top P_g {\neq} \overline b_g$}, where $\overline{b}_g:=\max_{i\in I} b_{gi}$.
\medskip{}
\item $(\alpha (I+\mathds{1}  \mathds{1}^\top)  +Q_g )^{-1}(\beta \ones-b_g)\in \mathbb{R}^n_{>0}$.
\medskip{}
\item {The inequality}
\be\label{e:upperbound}
\ones^\top (\alpha (I+\mathds{1}  \mathds{1}^\top)  +Q_g )^{-1}(\beta \ones-b_g) < {\frac{\beta-\overline b_g}{\alpha}},
\ee
holds.
\end{enumerate}
\end{lem}

\begin{proof}
See Appendix.
\end{proof}

Summarizing, Corollary \ref{first.order.optimality.condition.lq} and Proposition \ref{first.order.optimality.condition.game.lq} identify the optimal production and optimal demand, respectively, for the cost and utility functions \eqref{e:Cg-LQ} and \eqref{e:U-LQ}. Lemma \ref{l:equivalence-demand} characterizes the conditions under which the optimal demand of each consumer is strictly positive, and Lemma \ref{cor:opt.Pg} 
 provides equivalent conditions for strict positivity of optimal productions. 
{Based on the aforementioned results, the following necessary and sufficient condition for the 
existence of an optimal triple $(P_g^\star, P_d^\star, p^\star)\in \R^n_{>0}\times \R^m_{>0}\times \R_{>0}$ can be given:}

\bigskip{}
{
\begin{theorem}\label{p:optimal}
Let the price function admit the affine form $p(P_g)=\beta-\alpha\ones^\top P_g$, for some positive scalars $\alpha, \beta>0$.
For linear-quadratic functions \eqref{e:Cg-LQ}, \eqref{e:U-LQ}, 
let $P_g^\star$ denote the Cournot-Nash equilibrium solution to \eqref{game.lq}, and $P_d^\star$ the optimal demand solution to  \eqref{ut.max.pbm} computed with respect to the optimal price $p^\star:=p(P_g^\star)$. Then, the following are equivalent:

\begin{enumerate}
\item \be\label{suff.cond}
\frac{\beta-\underline{b}_d}{\alpha} {<} \ones^\top (\alpha (I+\mathds{1}  \mathds{1}^\top)  +Q_g )^{-1}(\beta \ones-b_g) {<} {\frac{\beta-\overline b_g}{\alpha}}.
\ee
\item 
\bse\label{e:opt-explicit}
\begin{align}
\label{e:Pg-opt}
P_g^\star&=(\alpha (I+\mathds{1}  \mathds{1}^\top)  +Q_g )^{-1}(\beta \ones-b_g)\\
\label{e:Pd-opt}
P_d^\star&={Q}_d^{-1} ({b}_d - \beta \ones +\alpha \ones  \ones^\top P^\star_g)\\
\label{e:p*}
p^\star&=\beta-\alpha \ones^\top P^\star_g,
\end{align}
\ese
and $\overline b_g\neq p^\star\neq \underline{b}_d.$ 
\medskip{}
\item $(P_g^\star, P_d^\star, p^\star)\in \mathbb{R}^n_{>0}\times \mathbb{R}^m_{>0}\times \mathbb{R}_{>0}$.
\end{enumerate}
\medskip{}
Moreover, in case $\alpha=\alpha^*$ and $\beta=\beta^*$ with $\alpha^*$ and $\beta^*$ as in \eqref{e:alpha-beta}, then the balancing condition holds, namely
\be\label{e:balanced-price}
\ones^\top P_g^\star = \ones^\top P_d^\star.
\ee
\end{theorem}
\begin{proof}
By Lemma \ref{l:equivalence-demand} and Lemma \ref{cor:opt.Pg}, establishing the equivalence of the three statements is straightforward.
{Now suppose that $\alpha=\alpha^*$ and $\beta=\beta^*$. Then equation \eqref{e:Pd-opt} becomes
\[
P_d^\star={Q}_d^{-1} ({b}_d - \beta^* \ones +\alpha^* \ones  \ones^\top P^\star_g).
\]
This yields 
\[
\ones^\top P_d^\star= \ones^\top Q_d^{-1}b_d-\beta^* \ones^\top Q_d^{-1}\ones +\alpha^*  \ones^\top Q_d^{-1}\ones\ones^T P_g^\star.
\]
The balancing condition \eqref{e:balanced-price} then follows from \eqref{e:alpha-beta}.}
\end{proof} 
}
\begin{remark}\label{r:equivalence}
{
\nmo{The vector}
$P_g^\star$ {in \eqref{e:Pg-opt}} can be rewritten as
\be\label{straight.manip}\ba{rl}
P_g^\star&=(\alpha I+Q_g )^{-1}(\beta \ones-b_g-\alpha \ones\ones^\top   P_g^\star)\\
&=(\alpha I+Q_g )^{-1}(\ones p^\star-b_g).
\ea\ee
Hence,}
the triple $(P_g^\star, P_d^\star, p^\star)$ can be equivalently characterized by the implicit form
\bse\label{e:opt-implicit}
\begin{align}
P_g^\star&= (\alpha I+Q_{g})^{-1} (\ones p^\star- b_{g}),\\
\label{e:demand-implicit}
P_d^\star&={Q}_d^{-1} ({b}_d - \ones p^\star),\\
p^\star&= \beta-\alpha  \ones^\top P^\star_g.
\end{align}
\ese
\end{remark}

\begin{remark}\label{r:interpret}
To  give an interpretation to condition \eqref{suff.cond}, we rewrite it in a different form. 
By Theorem \ref{p:optimal}, condition \eqref{suff.cond} can be rewritten as 
\be\label{suff.cond.2}
\frac{\beta-\underline{b}_d}{\alpha} <\ones^\top P_g^\star < {\frac{\beta-\overline b_g}{\alpha}},
\ee
which is equivalent to $\overline{b}_g < \beta - \alpha \ones^\top {P_g^\star}< \underline{b}_d$, or also
\[
C'_{i}(0) < p(P_g^\star) < U'_{i}(0), \quad \forall i\in \mathcal{I}. 
\]
Noting that $p(P_g)$ is monotonically decreasing, the lower bound yields $C'_{i}(0) < p(0)$.
This means that
the marginal costs of the producers are lower than the price at zero generation ($P_{gi}=0$). 
Therefore, the producers always benefit from providing nonzero amount of goods to the consumers. 
Analogously, the upper bound indicates that the marginal utility of each consumer at zero demand ($P_{dj}=0$) is higher than the 
the eventual optimal price $p^\star=p(P_g^\star)$ dictated by the consumers.
Hence, \cdpter{under condition \eqref{suff.cond},} it is always advantageous for consumers to enter the market and have a strictly positive demand.
\end{remark}

\section{Cournot-Nash optimal dynamical networks}\label{sec:3}

In the previous section, we studied Cournot competition and characterized the Cournot-Nash equilibrium \cdp{among}
producers and consumers.
In this section, {we introduce a dynamical network whose output variables are affected by the cumulative effect of demand and generation mismatch. Using these variables as measurements, we}
devise a dynamic 
output feedback algorithm
that steers the dynamical network to the Cournot-Nash optimal solution identified by the triple $(P_g^\star, P_d^\star, p^\star)$ {with
\bse\label{e:opt-star-version}
\begin{align}
\label{e:Pg-opt-star}
P_g^\star&=(\alpha^* (I+\mathds{1}  \mathds{1}^\top)  +Q_g )^{-1}(\beta^* \ones-b_g)\\
\label{e:Pd-opt-star}
P_d^\star&={Q}_d^{-1} ({b}_d - \beta^* \ones +\alpha^* \ones  \ones^\top P^\star_g)\\
\label{e:p*-star}
p^\star&=\beta^*-\alpha^* \ones^\top P^\star_g.
\end{align}
\ese
Note that the triple above is obtained from \eqref{e:opt-implicit} by setting $\alpha=\alpha^*$ and $\beta=\beta^*$. 
The reason why we are interested in the latter choice is to ensure that the balancing condition \eqref{e:balanced-price} is met.}
Moreover, we require 
\cdp{a reduced amount of}
information \cdp{about} the consumers to allow for the changing demand present in dynamic interactive markets.   
Finally, note that the feedback algorithm should be designed such that the stability of the physical system is not compromised.
\subsection{Network dynamics}
The topology of the network is represented by a connected and undirected graph $\calG(\calV,\calE)$ with a vertex set $\calV=\{1, \ldots, n\}$, and an edge set $\calE$ given by the set of unordered pairs $\{i, j\}$ of distinct vertices $i$ and $j$.  The cardinality of $\calE$ is denoted by $m$\footnote{The integer $m$ should not be confused with the number of consumers in the previous section, as the latter is equal to the number of producers in this  section and is denoted by $n$.}. 
The set of neighbors of node $i$ is denoted by $\calN_i=\{j\in \calV \mid {\{i,j\}\in \calE}\}.$
%

We 
consider 
a second-order consensus based network dynamics of the form:
\begin{align}\label{e:2order}
m_i\ddot x_{i}+d_i\dot x_i \cdp{-} \sum_{j\in \calN_i} \nabla H_{ij} (x_{i}-x_{j})=u_i, 
\end{align}
where 
$m_i,d_i\in \R_{>0}$ 
are constant, $H_{ij}: \R \rightarrow \R$ is a continuously differentiable strictly convex\footnote{\nm{\cdpter{Locally} strictly convex functions can be analogously treated in the analysis, with the only difference that the convergence result will become local in this case.}}with its minimum at the origin, 
$\nabla H_{ij}$ denotes the partial derivative of $H_{ij}$ with respect to its argument $x_i-x_j$,
$x_i\in \R$ is the state associated to node $i$, and $u_i\in \R$ is  the input applied to the dynamics of the $i$th node. Note that as $H_{ij}$ is strictly convex, $\nabla H_{ij}$ 
is a strictly increasing function of $x_i-x_j$. 
The second order dynamics \eqref{e:2order} can represent 
\cdp{different kind of networks, including}
power networks, by appropriately choosing the function $H_i$, see e.g. \cite{van2013port,kundur94}. 
Producers and consumers affect the dynamics \eqref{e:2order} via
\[
u_i=P_{gi}-P_{di},
\]
where $P_{gi}$ is the production and $P_{di}$ is the aggregated demand at node $i$ as before.
This means that a mismatch between production and demand causes the 
\cdp{node $i$}'s state variable to drift away from its unforced behavior. 
Note that the number of producers and consumers here are considered to be the same, and we thus use the notation $P_{di}$ rather than $P_{dj}$ which was used in the previous section.

%

Let $R$ be the incidence matrix of the graph $\calG$. Note that, by associating an arbitrary orientation to the edges, the incidence
matrix $R\in\R^{n \times m}$ is defined element-wise as $R_{ik} = 1,$ if node
$i$ is the sink of 
\cdp{edge $k$,}
$d_{ik} = -1,$ if $i$ is the source of 
\cdp{edge $k$,}
and $R_{ik} = 0$ otherwise. In addition, $\ker R^\top=\im \ones$ for a \cdp{connected} graph $\calG$.
Then, \eqref{e:2order} can be written in vector form as
\bse\label{e:2order-compact}
\begin{align}
\dot x&= y\\
M\dot y&=-D y - R\, \nabla H(R^\top x) + P_g- P_d
\end{align}
\ese
where $x=\col(x_i)$, $M=\bdiag(M_i)$, $D=\bdiag(D_i)$, $P_g=\col(P_{gi})$, $P_d=\col(P_{di})$, $i\in \calV$. In addition, $\col(\cdp{\nabla H_{ij}})$ is denoted by 
$\nabla H: \R^m \rightarrow \R^m$, where the edge ordering in $\nabla H$ is the same as that of the incidence matrix $R$.
It is easy to see that \eqref{e:2order-compact} has non-isolated equilibria for constant vectors $P_g$ and $P_d$. In fact, 
given a solution $(x,y)$ of \eqref{e:2order-compact}, $(x+c\ones, y)$ is a solution to \eqref{e:2order-compact} as well, for any constant $c\in \R$.
To avoid this complication, we perform a change of coordinates by defining  

\be\label{varphi}
\zeta_i = x_i -x_n, \quad i=1,\ldots, n-1.
\ee


\cdp{This new set of coordinates} can be also written as
\[
\begin{bmatrix}
\zeta_1\\
 \vdots \\ \zeta_{n-1} \\
 0
\end{bmatrix}=
\begin{bmatrix}
x_1 \\ 
\vdots \\ x_{n-1}\\ x_n
\end{bmatrix}-\mathds{1} x_n.
\]
Let $R_\zeta\in \R^{n-1} \times \R^m$ denote the incidence matrix with its $n$-th row removed. Then, by the equality above we have 
$$R^\top x=R_\zeta^\top \zeta,$$
where $\col(\zeta_i)=\zeta \in \R^{n-1}$.  
Moreover, it holds that $\zeta= E^\top x$ where $E^\top=\bbm I_{n-1} & -\ones_{n-1} \ebm$.
Noting that $R=ER_\zeta$, and defining a function $H_\zeta$ such that $H(R_\zeta^\top\zeta)=H_\zeta(\zeta)$, and thus 
$R_\zeta \nabla H(R_\zeta^\top\zeta)= \nabla H_\zeta(\zeta)$,
the system \eqref{e:2order-compact} in the new coordinates reads as 
%
\bse\label{e:delta-compact}
\begin{align}
\dot\zeta &=E^\top y\\
M \dot y&=-Dy -E \,\nabla H_\zeta (\zeta)+P_g-P_d,
\end{align}
\ese

\begin{remark}
\cdpter{
Note that the system above belongs to the class of dynamical networks  given by the feedback interconnection of incrementally output feedback passive dynamics at the nodes of the form  \cite{arcak}} 
\nmo{
\be\label{e:inc-passive-1}
\ba{rl}
\dot y&= f(y,v, w)\\
z &=  h(y) 
\ea
\ee
with $h(y)=y$, $f(y,v, w) = M^{-1}(-Dy+v+w)$, $w=P_g-P_d$, and incrementally passive dynamics at the edges 
\be\label{e:inc-passive-2}
\ba{rl}
\dot\zeta &=\mu \\
\sigma &= \nabla H_\zeta (\zeta).
\ea
\ee 
where the interconnection constraints are given by
\[
\mu = E^\top y, \quad v= - E \sigma.
\]
We have opted to consider the system \eqref{e:delta-compact} rather than more general subclasses of \eqref{e:inc-passive-1}, \eqref{e:inc-passive-2}, to keep the focus of the paper and provide more explicit results.}
\end{remark}


As a result of the change of coordinates, the network \cdpter{\eqref{e:delta-compact}} now has at most one equilibrium, for given constant vector $P_g$ and $P_d$, and we have the following lemma:

\begin{lem}\label{l:equib}
Let $P_g=\overline P_g$ and $P_d=\overline P_d$ for some constant vectors $\overline P_g, \overline P_d\in \R^n$.  
Then the point $(\overline \zeta, \overline y)$  is an equilibrium of \eqref{e:delta-compact} if and only if
\begin{align}\label{e:y*}
&\overline y=\ones y^*, \qquad y^*=\frac{\sum_{i\in \cal V} \cdp{(\overline P_{gi}-\overline P_{di})}}{\sum_{i\in \calV}  D_i},
\end{align}
\begin{align}\label{e:feas0}
&\nabla H_\zeta (\overline\zeta) =  E^+(I_n - \frac{D\ones \ones^\top}{\ones^\top D\ones}) (\overline P_g- \overline P_d),
\end{align}  
where $E^+=(E^\top E)^{-1}E^\top$. 
Moreover, the equilibrium, if exists, is unique. 
\end{lem}

\begin{proof} See Appendix.
\end{proof}

\subsection{Dynamic pricing mechanism}

Next, we seek for a dynamic feedback (pricing mechanism) that steers the physical network to 
\cdp{an asymptotically stable equilibrium, while guaranteeing the convergence of the production and the demand to the Cournot-Nash solution.}
In particular, we are interested to regulate the production $P_g$ to $P_g^\star$, the demand $P_d$ to $P_d^\star$,  and attain \nmo{the optimal} price $p^\star$ given by \eqref{e:opt-star-version}.  Note that, \cdp{statically setting the generation and  production as} 
$P_g=P_g^\star$ and $P_d=P_d^\star$  is undesirable as it requires complete information of the entire network and utility functions. 

Recall that in the Cournot model, consumers are price takers, meaning they optimize their utility functions given a price. Consistent  with 
\eqref{e:Pd-opt-star}-\eqref{e:p*-star}
we consider the demand as
\be\label{eq:demand} 
P_{di}\cdp{(t)}={Q}_{di}^{-1} ({b}_{di} - p_i(t)),
\ee
%
where $p_i(t)$ can be interpreted as a {\em{momentary}} or {\em{estimated}} price for the $i$th consumer at time $t$. In vector form this is written as
$P_d={Q}_d^{-1} ({b}_d - p(t))$, with $p(t)=\col(p_i(t))$.

Next, looking at the expression of $y^*$ in Lemma \ref{l:equib}, we notice that the deviation from the supply-demand matching condition \eqref{balance0}
 is reflected on the steady-state value of the state variable $y$. In fact, \eqref{balance0}
 holds if and only if $y^*=0$. This motivates the implementation of a negative feedback from $y$ to $P_g$ in the controller.
Moreover, in order to ensure optimality, we rely on a communication layer next to the physical network that appropriately distribute the information on the local price estimations $p_i$ over the entire network. The topology of this communication layer is 
\cdp{modeled via}
an undirected connected graph $\calG_c(\calV, \calE_c)$, and the set of neighbors of the 
\cdp{node $i$}
is denoted by $\calN^c_i$. Inspired by the aforementioned remarks, and motivated by some additional stability and optimality considerations, the following distributed controller (pricing mechanism) is  proposed:
\bse\label{e:controller}
\begin{align}
\nm{\tau_i}\dot p_i&=- k_i y_i - {Q}_{di}^{-1}y_i - \sum_{j\in \calN^c_i} \rho_{ij} (p_i-p_j)\\[-2mm]
P_{gi}&=k_i(p_i- b_{gi}) \label{eq:gen}
\end{align} 
\ese
where \nm{$\tau_i>0$} is the time constant, $\rho_{ij}>0$ indicates the weight of the communication at each link, and the constant parameter $k_i>0$ will be specified later. \nm{Let $T=\diag(\tau_i)$, $K=\diag(k_i)$, and the weighted Laplacian matrix of $\calG_c$ be denoted by $L$.
Then the} overall closed-system admits the following state-space representation:
\bse\label{e:CL}
\begin{align}
\dot\zeta &=E^\top y\\
\nonumber
M \dot y&=-Dy -E \,\nabla H_\zeta (\zeta)\\
&\qquad  + K(p- b_g)  - {Q}_d^{-1} ({b}_d -p)\\
\nm{T}\dot p&=- L p - K  y - {Q}_d^{-1}y
\end{align}
\ese 

The result below characterizes the static properties of the closed-loop system \eqref{e:CL}.
\begin{proposition}\label{l:CL-static}
The point $(\overline \zeta, \overline y, \overline p)$  is an equilibrium of \eqref{e:CL} if and only if 
$\overline y=0$,
\be\label{e:price-nonoptimal}
\overline p=\ones_n q, \qquad q=\frac{\ones^\top Kb_g+\ones^\top Q_d^{-1}b_d}{\ones^\top K\ones+\ones^\top Q_d^{-1}\ones},
\ee
and $\overline \zeta$ satisfies 
\be\label{e:feas}
\nabla H_\zeta (\overline\zeta) = E^+(\overline P_g- \overline P_d)
\ee
with
$$\overline P_g= K(\overline p- b_g), \quad \overline P_d=  Q_d^{-1} ({b}_d -\overline p).$$
The equilibrium, if exists, is unique.
Moreover, if $K=(\alpha^* I_n+Q_g)^{-1}$, then $(\overline P_g, \overline P_d, \cdpter{\overline p})=(P_g^\star, P_d^\star, p^\star)$ given by \eqref{e:opt-star-version}. 
\end{proposition}

 \begin{proof}
  Suppose that $(\overline \zeta, \overline y, \overline p)$ is an equilibrium of \eqref{e:CL}. Then, 
 \bse\label{e:CL-proof}
\begin{align}
0 &=E^\top \overline y\\
\nonumber
0&=-D\overline y -E \,\nabla H_\zeta (\overline \zeta)\\
&\qquad  + K(\overline p- b_g)  - {Q}_d^{-1} ({b}_d - \overline p)\\
\label{e:Cl-equib-controller-proof}
0&=- L \overline p - K \overline  y - {Q}_d^{-1}\overline y.
\end{align}
\ese 
By the first equality, we have $\overline y=\ones y^*$ for some $y^*\in \R$. Substituting this into \eqref{e:Cl-equib-controller-proof}, and multiplying both sides of \eqref{e:Cl-equib-controller-proof} from the left by $\ones^\top$, we obtain that $y^*=0$. Hence, \eqref{e:Cl-equib-controller-proof} results in 
$\overline p=\ones  q $ for some $q\in \R$. The fact that $q$ is given by \eqref{e:price-nonoptimal}, and that \eqref{e:feas} holds, follow  from suitable algebraic manipulations analogous to the proof of Lemma \ref{l:equib}. The converse result as well as uniqueness of the equilibrium also follow analogous to Lemma \ref{l:equib}.

By \eqref{e:price-nonoptimal}, we have
\[
(1+\alpha^* \ones^\top K \ones ) q= \beta^*+ \alpha^* \ones^\top K b_g
\]
where $\alpha^*$ and $\beta^*$ are given by \eqref{e:alpha-beta}. The equality above can be written as
$
q= \beta^*- \alpha^* \ones^\top K(\ones q- b_g),
$
which yields 
\be\label{e:q-implicit}
q=\beta^*- \alpha^* \ones^\top \overline P_g.
\ee
Equivalently, we have 
$K^{-1}K (\ones q - b_g)+ b_g= \ones\beta^*- \alpha^* \ones\ones^\top \overline P_g,$
and hence
\[( \alpha^* \ones\ones^\top + K^{-1}) \overline P_g= \ones\beta^*- b_g.\]
By setting $K=(\alpha^* I_n+Q_g)^{-1}$, the equality above returns $\overline P_g=P_g^\star$, with $P_g^\star$ given by \eqref{e:Pg-opt}.
Then, by comparing \eqref{e:q-implicit} to \eqref{e:p*}, we find that $q=p^\star$. Finally, 
\[
\overline P_d=Q_d^{-1} (b_d- \ones q)= Q_d^{-1} (b_d- \ones \beta^* +\alpha^* \ones \ones^\top P_g^\star  )=P_d^\star, 
\]
where the second equality follows from \eqref{e:q-implicit}, and the last one from \eqref{e:Pd-opt}.
 \end{proof}
 
 \begin{remark}
In case of linear dynamics, namely $2H_\zeta(\zeta)=\zeta^\top R_\zeta W  R_\zeta^\top  \zeta$, $W>0$, the vector $\overline \zeta$ is explicitly obtained as
\[
\overline \zeta = (R_\zeta W R_\zeta^\top)^{-1}  E^+ (\overline P_g- \overline P_d).
\]
\end{remark}
\medskip{}

\cdpter{Proposition}  \ref {l:CL-static} imposes the following assumption: 

\begin{assum}\label{a:feas}
There exists $\overline \zeta\in \R^{n-1}$ such that \eqref{e:feas}
is satisfied.
\end{assum}

\begin{remark}
As evident from \eqref{e:feas0},  the condition in Assumption \ref{a:feas} is a consequence of the agents' dynamics \eqref{e:2order}, rather than the choice of the controller.
In case the graph $\calG$ is a tree, the incidence matrix has full column rank and Assumption \ref{a:feas} is always satisfied with
\[
\overline \zeta=  (R_\zeta ^+)^\top  \nabla H ^{-1} \big((R^\top R)^{-1} R^\top (\overline P_g- \overline P_d)\big)
\]
where $R_\zeta ^+$ is a right inverse of $R_\zeta$, and $\nabla H( \nabla H^{-1} (x))=x, \forall x\in \R^m.$
\end{remark}

The next theorem provides the main result of this section, which validates the proposed feedback algorithm.  

\begin{theorem}\label{t:main-dynamic}
Let Assumption \ref{a:feas} hold. 
Then, the equilibrium $(\overline\zeta, \overline y, \overline p)$ of \eqref{e:CL} is asymptotically stable. 
Moreover, for $K= (\alpha^* I_n+Q_g)^{-1}$, the vector $(P_g, P_d, p)$, \cdpter{with $P_g$ defined as in \eqref{eq:gen} and $P_d$ as in \eqref{eq:demand},} 
 asymptotically converges to the optimal Cournot-Nash solution $(P_g^\star, P_d^\star, p^\star)$, \cdp{the latter} given by \eqref{e:opt-star-version}. 
\end{theorem}

\begin{proof}
To prove asymptotic stability, we consider the Lyapunov function candidate 
\begin{align*}
V=\frac{1}{2} (y-\overline y)^\top M(y-\overline y)+ \frac{1}{2}( p -\overline p)^\top\nm{T}(p-\overline p)+ \calH(\zeta)
\end{align*}
where $\calH$ takes the form of the Bregman distance {between $\zeta, \overline \zeta$ associated with the {\it distance-generating}} function $H(\zeta)$ and the point $\overline \zeta$, namely \cite{bregman1967relaxation}
\[
\calH(\zeta)= H(\zeta)-H(\overline \zeta)- (\zeta- \overline \zeta)^\top \displaystyle \left.\frac{\partial H}{\partial \zeta}\right|_{\overline \zeta} .
\]
Since $H$ is strictly convex, the Bregman distance $\calH$ is nonnegative and is equal to zero whenever $\zeta=\overline \zeta$. Then, clearly the function $V$ has a strict minimum at $(\overline \zeta, \overline y, \overline p)$. Computing the time derivative of $V$ along the solutions of \eqref{e:CL} yields
\[
\dot V= -(y-\overline y)^\top D(y- \overline y) - (p -\overline p)^\top L(p-\overline p),
\]
where we have used \eqref{e:CL-proof} together with the fact that $$\frac{\partial{\calH}}{\partial{\zeta}}=\left.\frac{\partial{H}}{\partial{\zeta}}-\frac{\partial{H}}{\partial{\zeta}}\right|_{\overline \zeta} .$$ As $V$ is positive definite, and $\dot V$ is nonpositive, we conclude that solutions of \eqref{e:CL} are bounded.
By invoking  LaSalle's invariance principle, on the invariant set we have
$
y=\overline y, \quad Lp=L\overline p.
$
Noting that $\overline y=0$ and $L\overline p=0$, we find that each point on the invariant set is an equilibrium of \eqref{e:CL}. By Proposition \ref{l:CL-static}, the equilibrium is unique, and therefore the invariant set comprises only the equilibrium point $(\overline \zeta, 0, \overline p)$ given by \eqref{e:price-nonoptimal} and \eqref{e:feas}. 
By continuity, the vectors $P_g$ and $P_d$ asymptotically converge to $\overline P_g= K( \overline p- b_g)$ and $\overline P_d=  Q_d^{-1} ({b}_d - \overline p)$, respectively. For $K=(\alpha^* I_n+Q_g)^{-1}$, by Proposition \ref{l:CL-static} we have $(\overline P_g, \overline P_d, \overline p)=(P_g^\star, P_d^\star, p^\star)$. This completes the proof. 
\end{proof}

\begin{remark}
By Theorem \ref{t:main-dynamic}, the controller \eqref{e:controller} with $k_i= (\alpha^* + Q_{gi})^{-1}$, steers the network to the Cournot-Nash optimal solution. 
Note that,  with the exception of the parameter $\alpha^*$, the $i$-th controller uses only the local variables at node $i$ together with the communicated variables 
$p_i-p_j$  of the neighboring nodes in the communication graph. If the parameter $\alpha^*$ is not precisely known, then $k_i$ is set to $(\hat\alpha^*_i + Q_{gi})^{-1}$ where $\hat\alpha^*_i$ is an approximation of $\alpha^*$ at node $i$. This approximation will shift the equilibrium of the closed-loop system away from the one associated with the Cournot-Nash solution. However, by  Theorem \ref{t:main-dynamic}, asymptotic stability will not be jeopardized, local price variables will synchronize, and the vector $(P_g, P_d, p)$ will converge to the point 
$(\overline P_g, \overline P_d, \overline p)$ given in Lemma \ref{l:CL-static}. The investigation of how far the equilibrium is from the Cournot-Nash equilibrium in the presence of uncertainty on $\alpha^*$ is left for future research. 
\end{remark}


\begin{remark}
When $\alpha^*$ is not precisely known, another possibility is to estimate the parameter 
in advance, 
as utility functions are not frequently changing.
To this end, one can implement a distributed algorithm such as
\bse\label{e:estimator}
\begin{align}
\dot\chi_{ij}&= \hat\alpha_i- \hat\alpha_j, \qquad \{i, j\} \in E_c \\
\dot{\hat\alpha}_i&= \frac{1}{n}- Q_{di}^{-1}  \hat\alpha_i - \sum_{j\in \calN_i^c}  \kappa_{ij} \chi_{ij}, \quad i\in \calV
\end{align} 
\ese
which requires local parameter $Q_{di}$, communicated variables $\chi_{ij}$, and assumes that each controller is aware of the total number of participating agents, namely $n$.
It is easy to see that $\hat\alpha$ asymptotically converges to $\alpha^*=(\ones^\top Q_d^{-1}\ones)^{-1}$.
\nm{While it is difficult to provide analytical guarantees for the online use of this estimator in the controller \eqref{e:controller}, our numerical investigation in Section \ref{s:simulation} validates stability and performance of such a scheme.} 
\end{remark}

\section{Case study}\label{s:simulation}
We illustrate the proposed pricing mechanism on a \cdpter{specific example of a network with producers and consumers, namely a} four area power network \cite{trip-et-al-aut16}, see \cite{nabavi2013topology} on how a four area network equivalent can be obtained for the IEEE New England 39-bus system or the South Eastern Australian 59-bus.
The power network model we consider here is given by the so-called swing equation \cite{kundur94}, and is mathematically equivalent to the dynamics in \eqref{e:delta-compact}, under the assumption that voltages are constant and the frequency dynamics is decoupled from the reactive power flow. In this case, $\zeta$ is the vector of phase angles measured with respect to the phase angle of a reference bus (area $4$), $y$ is the vector of frequency deviations from the nominal frequency ($50/60$Hz), and the diagonal matrices $M$ and $D$ collect the inertia and damping constants. The vectors $P_g$ and $P_d$ denote the vector of generation and demand as before. The numerical values of the system parameters are provided in Table \ref{t:param}. The physical and communication graphs, namely $\calG$ and $\calG_c$, are depicted in Figure \ref{f:numerical}, where the solid and dotted edges denote the transmission lines and communication links, respectively.

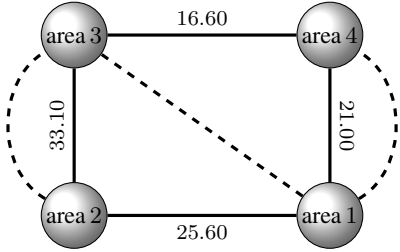
\begin{figure}
\[\begin{tikzpicture}[x=1.7cm, y=1.2cm,
    every edge/.style={sloped, draw, line width=1.2pt}]


\vertex (v3) at (2,1)  {\small \,area$\,3\,$};
\vertex (v1) at (4,-1) {\small \,area$\,1\,$};
\vertex (v2) at (2,-1) {\small \,area$\,2\,$};
\vertex (v4) at (4,1) {\small \,area$\,4\,$};
%
\path
(v1) edge node[anchor=north]{\weight{\rm 25.60}}(v2)
(v2) edge node[anchor=south]{\weight{\rm 33.10}}(v3)
(v3) edge node[anchor=south]{\weight{\rm 16.60}}(v4)
(v4) edge node[anchor=south]{\weight{\rm 21.00}}(v1)
(v1) edge[dashed] node[left]{}(v3)
(v2) edge[dashed, bend left,in=120,out=60] node[left]{}(v3)
(v1) edge[dashed, bend left,in=-120,out=-60] node[left]{}(v4)
                ;
\end{tikzpicture}\]
\caption{The solid lines denote the transmission lines, and the dashed lines depict the communication links. The edge weights indicate the susceptance of the transmission lines.}\label{f:numerical}
\end{figure}

For each $\{i,j\}\in \calE$, the (locally convex) function $H_{ij}$ in \eqref{e:2order} is given by $-|B_{ij}|V_iV_j\cos(x_i-x_j)$, where $B_{ij}<0$ is the susceptance of the line $\{i,j\}$, and $V_i$ and $x_i$ are the voltage magnitude and voltage phase angle at the $i$th area (bus).
In \eqref{e:delta-compact}, this yields the expression 
\[
\nabla H= R_\zeta\nmo{W} \boldsymbol{\sin}(R_\zeta^\top\xi),
\]
where $W=\diag(w_k)$, with $w_k:=B_{ij}V_iV_j$, $k\sim \{i, j\}$, and $\boldsymbol{\sin}(\cdot)$ is interpreted elementwise.


\begin{figure}[t!]
\hspace*{-0.2cm}
\includegraphics[width = 9cm]{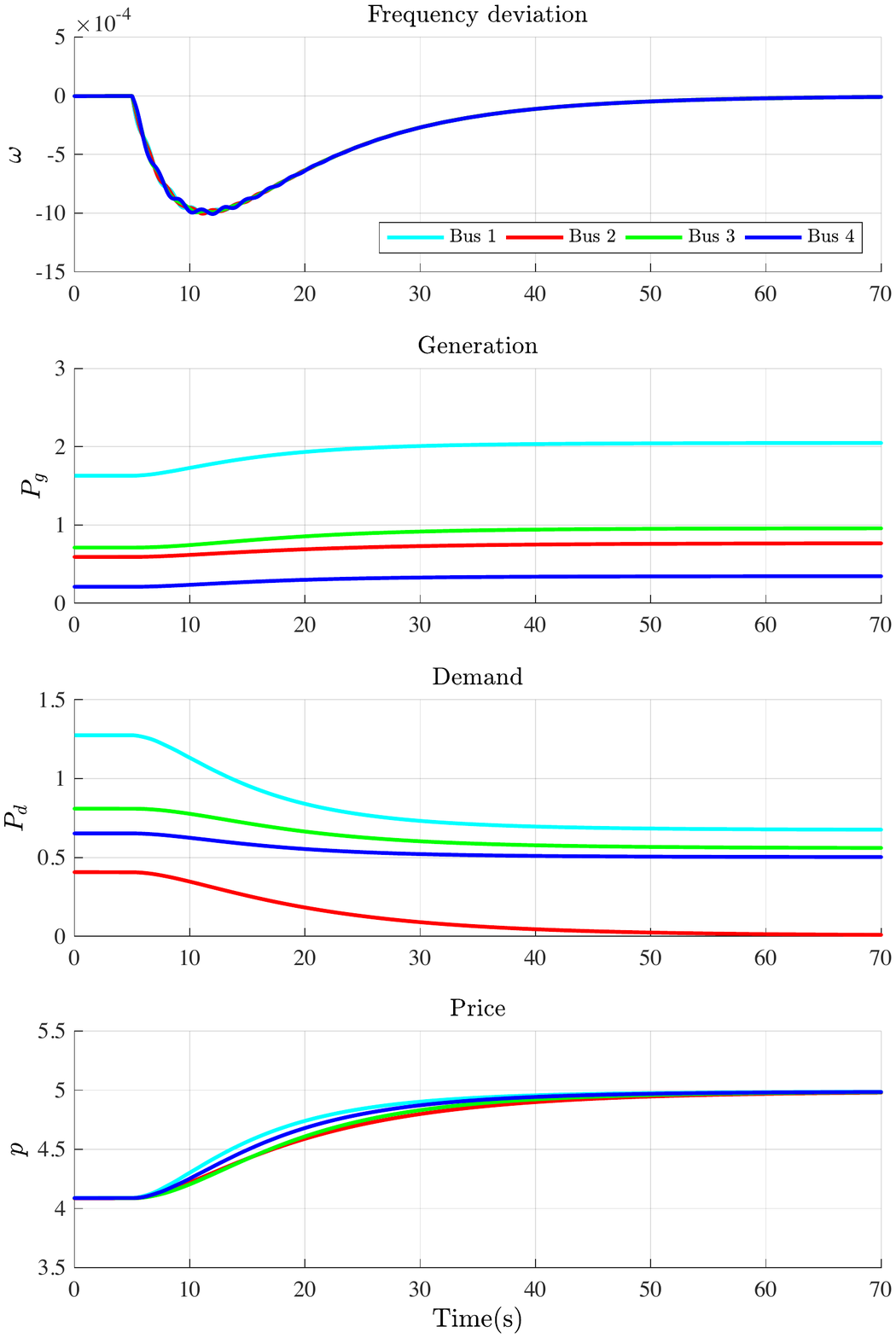}
\caption{Numerical simulation of the closed-loop system \eqref{e:CL}.}
\label{f:nom}
\end{figure}

\begin{figure}[t!]
\hspace*{-0.2cm}
\includegraphics[width = 9cm]{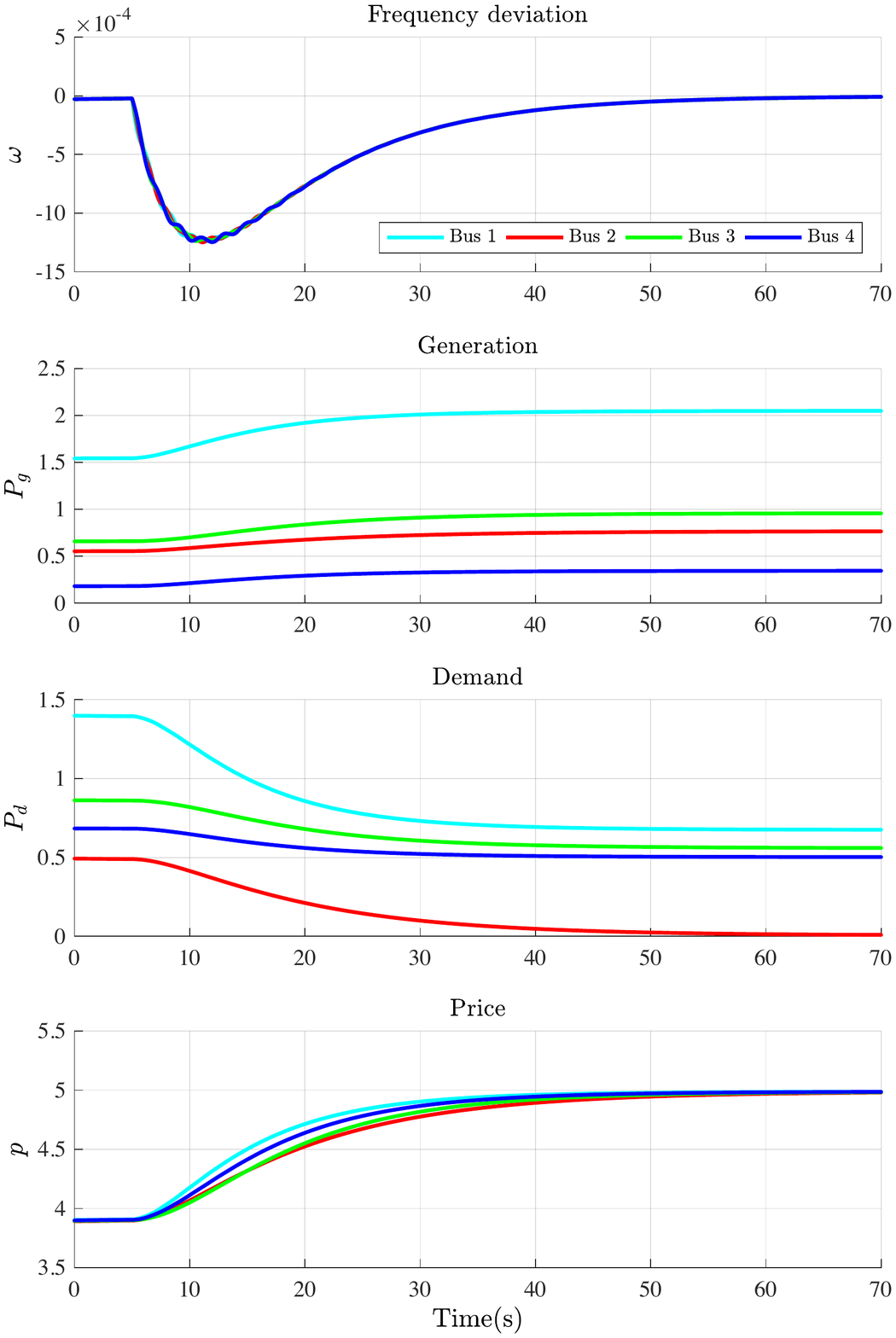}
\caption{Numerical simulation of the closed-loop system \eqref{e:CL-estimator}.}
\label{f:estimator}
\end{figure}
We consider linear-quadratic cost and utility functions given by \eqref{e:Cg-LQ}, \eqref{e:U-LQ}, with the parameters provided in Table \ref{t:param}.

\begin{table}[h!]
	\centering
	\caption{Simulation Parameters}
	\label{t:param}
	\scalebox{1}{
	\begin{tabular}{cccccccccccc}
		\toprule
		{\rm Areas} & 1 & 2 & 3 & 4\\[1.2mm]
		\midrule
		$M_i$ & 5.22 & 3.98 & 4.49 & 4.22\\[1.2mm]
		$D_i$ & 1.60 & 1.22 & 1.38 & 1.42 \\  [1.2mm]
		$Q_{gi}$ &  1.50 & 4.50 & 3.00 & 6.00\\[1.2mm]
		$b_{gi}$  &   0.60 & 1.05 & 1.50 & 2.70\\ [1.2mm]
		$Q_{di}$ &  1.50 & 2.25 & 3.60 & 6.00\\ [1.2mm]
		$b_{di}$  & 6.00 & 5.00 & 7.00 & 8.00\\[1.2mm]
		$\tau_i$ & 2.00 & 3.00 & 3.00 & 1.50 \\[1.2mm]
		
		\bottomrule
	\end{tabular}}
	\vspace{-0.3cm}
\end{table}
We consider the distributed controller in \eqref{e:controller}, and we set $\rho_{ij}=1$, for each $\{i, j\}\in \calE_c$.
The closed-loop system \eqref{e:CL} is initially at steady-state. At time $t=5$s, we modify the utility functions by increasing $b_d$ by $25$ percent, which results in a step in the demand. The response of the closed-loop system to this change is shown in Figure \ref{f:nom}, where the values are in per unit with respect to a base power of $1000$MVA.  As can be seen in the figure, at steady-state the frequency is regulated to its nominal value, which indicates that the matching condition \eqref{balance0} is satisfied. The local prices converge to the same value which identifies the market clearing price $p^\star$. As desired, the triple $(P_g, P_d, p)$ converge to the Cournot-Nash optimal solution 
\vspace*{-0.85cm}
\be\label{e:opt-numerical}
P_g^\star=\bbm
2.05\\
    0.77\\
    0.96\\
    0.34\\
\ebm
, \quad
P_d^\star=
\bbm
1.67\\
   0.56\\
   1.04\\
   0.83\\
\ebm
,\quad 
p^\star=4.99.
\ee
Next, we consider the case where the parameter $\alpha^*$ is unknown and is identified in real-time by the estimator \eqref{e:estimator}. 
This results in the closed-loop dynamics
\bse\label{e:CL-estimator}
\begin{align}
\dot\zeta &=E^\top y\\
\nonumber
M \dot y&=-Dy -E \,\nabla H_\zeta (\zeta)\\
&\qquad  + (\hat\alpha I_n+Q_g)^{-1}(p- b_g)  - {Q}_d^{-1} ({b}_d -p)\\
T\dot p&=- L p - (\hat\alpha I_n+Q_g)^{-1}  y - {Q}_d^{-1}y\\
\dot\chi&= R_c^\top\hat\alpha\\
\dot{\hat\alpha}&= \frac{1}{n}\ones- Q_{d}^{-1}  \hat\alpha - R_c \chi, 
\end{align} 
\ese 
where $R_c$ denotes the incidence matrix of the communication graph, and we have set $\kappa_{ij}=1$ for simplicity.
The system is initially at steady-state. At time $t=5$s, we modify as before the utility functions by increasing $b_d$ by $25$ percent.
At the same time, we decrease the elements of $Q_d$ by $20$ percent, which modifies the actual value of $\alpha$ according to \eqref{e:alpha-beta}.
For a better comparison to the system without the estimator, the initial value of $Q_d$ is chosen such that its new value will be equal to the one provided in Table \ref{t:param}. The response of the closed-loop system \eqref{e:CL-estimator} is illustrated in Figure \ref{f:estimator}.
As can be seen from the figure, frequency is regulated to its nominal value and the triple $(P_g, P_d, p)$ converge again to the one given by \eqref{e:opt-numerical}. This means that the controller \eqref{e:controller} equipped with the estimator \eqref{e:estimator} is able to steer the network to the Cournot-Nash optimal solution. Note that compared to Figure \ref{f:nom}, the transient performance is only slightly degraded. 

\nmo{
\section{Conclusions}\label{sec:4}
We have proposed a distributed feedback algorithm that steers a dynamical network to a prescribed equilibrium corresponding to the so-called Cournot-Nash equilibrium. We characterized this equilibrium for linear-quadratic utility and cost functions, and specified the algebraic conditions under which the production and the demand are strictly positive for all agents. 
For a class of passive nonlinear second-order systems, where production and demands act as external inputs to the systems, we devised a control algorithm (pricing mechanism) that guarantees the convergence of the closed loop system to the optimal equilibrium point
associated with the previously characterized Cournot-Nash equilibrium.
Considering a different type of competition such as Bertrand and Stackelberg games, as well as a thorough comparison of the equilibrium points resulting from  competitive games against those obtained from  a social welfare problem are of interest for future research.}

\section*{Appendix}

\begin{pfof}{Lemma \ref{l:equivalence-demand}}
{From Corollary \ref{first.order.optimality.condition.lq}, it follows that the first two statements are equivalent, and they imply the third statement. It remains to show that {\it 3)$\Rightarrow$2)}.
Now, suppose that the third statement holds. From $P_d=Q_d^{-1}(b-\ones p)$,  using again Corollary \ref{first.order.optimality.condition.lq}, we obtain that $p\leq b_{dj}$ for all $j\in J$. The condition $p\neq \underline b_d$ yields $p<\underline b_d$, which completes the proof.}
\end{pfof}

\begin{pfof}{Lemma \ref{cor:opt.Pg}}
{{\it 1)$\Rightarrow$2)} Suppose that the first statement holds. Then, by Proposition \ref{first.order.optimality.condition.game.lq}, we have
\be\label{e:Pg-first-item-lem}
P_g= (2 \alpha I + Q_g )^{-1}(\beta \ones-b_g + \alpha (I - \mathds{1}  \mathds{1}^\top) P_g). 
\ee
This is equivalent to
$
\alpha P_g+\alpha \mathds{1}  \mathds{1}^\top P_g + Q_g  P_g = \beta \ones-b_g,
$
and to
$
P_g= (\alpha (I+\mathds{1}  \mathds{1}^\top)  +Q_g )^{-1}(\beta \ones-b_g).
$
Hence, {\eqref{Pg.opt} is obtained. \cdpter{This also shows the uniqueness of the solution}.  
Now, note that
$
(\alpha (I+\ones \ones^\top) +Q_g) P_g=\beta \ones - b_g,
$
which is equivalent to
$
(\alpha I+Q_g)  P_g= \ones (\beta -\alpha \ones^\top P_g) - b_g.
$
Element-wise, this can be written as
$
(\alpha+ Q_{gi}) P_{gi}= \beta -\alpha \ones^\top P_g - b_{gi}.
$ Since $P_{gi}>0$ for all $i$, we obtain that $\beta -\alpha \ones^\top P_g > \overline b_{g}$.
}
\medskip{}

{\it 2)$\Rightarrow$3)} Next, suppose that the second statement of the lemma holds. Then, using the same chain of equivalences as above, we obtain \eqref{e:Pg-first-item-lem}. Therefore, by Proposition \ref{first.order.optimality.condition.game.lq}, we must have that {$\gamma_i(P_{-gi})\cdpter{>} 0$ for all $i\in I$. Without loss of generality, assume that $b_{g1}=\overline b_g$. 
\cdpter{Suppose by contradiction that $\gamma_1(P_{-g1})\le  0$.} Then, we have $P_{g1}=0$, and thus  $\beta- \overline b_{g}-\alpha \ones^\top P_g=0.$ This contradicts the \nm{inequality}
in the second statement of the lemma, which completes this part of the proof.} 


{\it 3)$\Rightarrow$1)}  Now, let the third statement hold, and set 
\be\label{e:hatPg}
\hat P_g:= (\alpha (I+\mathds{1}  \mathds{1}^\top)  +Q_g )^{-1}(\beta \ones-b_g), \quad \hat P_g\in \R^n_{>0}.
\ee
Again, the vector $\hat P_g$ can be written in analogy to \eqref{e:Pg-first-item-lem} as $\hat P_{gi}= (2\alpha + Q_{gi})^{-1} (\beta - b_{gi} -\alpha \ones^\top \hat P_{-gi})$ for every component $i\in I$. 
For every $i\in I$,  since $\hat P_{gi}>0$, then $(2\alpha + Q_{gi})^{-1} (\beta - b_{gi} -\alpha \ones^\top \hat P_{-gi})>0$, which is equivalent to say that $\gamma_i(\cdpter{\hat P_{-gi}})>0$. 
\cdpter{Hence, the vector $\hat P_g$ satisfies \eqref{e:optimal.Pg}, and is therefore a solution to \eqref{game.lq} beloging to the interior of the positive orthant. } 
}  

{\it 3)$\Leftrightarrow$4)} To complete the proof of the lemma, it suffices to show that the last two statements of the lemma are equivalent, namely 
\[
\hat P_g\in \R^n_{>0} \Leftrightarrow \ones^\top (\alpha (I+\mathds{1}  \mathds{1}^\top)  +Q_g )^{-1}(\beta \ones-b_g) \le {\frac{\beta-\overline b_g}{\alpha}},
\]
where $\hat P_g$ is given by \eqref{e:hatPg}.
From \eqref{e:hatPg}, we have
$
(\alpha (I+\ones \ones^\top) +Q_g) \hat P_g=\beta \ones - b_g,
$
which is equivalent to
$
(\alpha I+Q_g) \hat P_g= \ones (\beta -\alpha \ones^\top \hat P_g) - b_g.
$
Element-wise, this can be written as
$
(\alpha+ Q_{gi})\hat P_{gi}= \beta -\alpha \ones^\top \hat P_g - b_{gi}.
$
Therefore, $\hat P_{gi}>0$ if and only if $\beta -\alpha \ones^\top \hat P_g> \overline{b}_g$. \cdpter{By replacing $\hat P_g$ with $(\alpha (I+\mathds{1}  \mathds{1}^\top)  +Q_g )^{-1}(\beta \ones-b_g)$, we see that} the latter inequality is equivalent to \eqref{e:upperbound}.
\end{pfof}

\begin{pfof}{Lemma \ref{l:equib}}
Suppose that $(\overline \zeta, \overline y)$  is an equilibrium of \eqref{e:delta-compact}. Then,
\bse
\begin{align}
0 &=E^\top \overline  y\\
\label{e:proof-equib-delta0}
0&=-D\overline y -E \,\nabla H_\zeta (\overline \zeta)+ \overline P_g- \overline P_d.
\end{align}
\ese
Hence, we find that $y=\ones_n \hat y$ for some $\hat y\in \R$, and
\be\label{e:proof-equib-delta}
0=-D\ones \hat y -E \,\nabla H_\zeta (\overline \zeta)+ \overline P_g- \overline P_d.
\ee
By multiplying both sides of the equality above from the left by $\ones^\top$, we find $\hat y=y^*$, where the latter is given by  \eqref{e:y*}.
By replacing the expression of $y^*$ back to the equality \eqref{e:proof-equib-delta},  and noting that $E$ has full column rank, the equality \eqref{e:feas0} is obtained.

Conversely, assume that  a point $(\overline \zeta, \overline y)$ satisfies \eqref{e:y*} and \eqref{e:feas0}. Clearly, $E^\top\overline y=0$.
Moreover, note that 
\[
(I_n - \frac{D\ones \ones^\top}{\ones^\top D\ones}) (\overline P_g- \overline P_d) \in (\im \ones_n)^\perp =\im E.
\]
Hence, multiplying both sides of \eqref{e:feas0} from the left  by $E$ gives
\[
E\nabla H_\zeta (\overline\zeta) =  (I_n - \frac{D\ones \ones^\top}{\ones^\top D\ones}) (\overline P_g- \overline P_d).
\]
By the definition of $\overline y$ in \eqref{e:y*}, the  equality above can be written as \eqref{e:proof-equib-delta0},
and therefore $(\overline \zeta, \overline y)$ is an equilibrium of \eqref{e:delta-compact}.
For uniqueness of the equilibrium, it suffices to show that
\[
\nabla H_\zeta (\overline \zeta)=\nabla H_\zeta (\tilde \zeta) {\text{ for some\;}} \overline\zeta, \tilde\zeta \in \R^{n-1}
\Longrightarrow \overline \zeta= \tilde{\zeta}.
\]
The equality $\nabla H_\zeta (\overline \zeta)-\nabla H_\zeta (\tilde \zeta)=0$ is equivalent to
\[
R_\zeta \nabla H (R_\zeta^\top \overline \zeta)- R_\zeta \nabla H (R_\zeta^\top \tilde \zeta)=0.
\]
Multiplying both sides of the equality above from the left by  $(\overline \zeta- \tilde{\zeta})^\top$ returns 
$
(R_\zeta^\top \overline \zeta - R_\zeta^\top \tilde{\zeta})^\top ( \nabla H (R_\zeta^\top \overline \zeta)- \nabla H (R_\zeta^\top \tilde \zeta))=0.
$
By strict convexity of $H$, we find that $R_\zeta^\top \overline \zeta=R_\zeta^\top \tilde \zeta$. The fact that $R_\zeta$ has full row rank yields  $ \overline \zeta= \tilde{\zeta}$, which completes the proof.
\end{pfof}

\bibliographystyle{IEEEtran}
\bibliography{nima2}

\end{document}